\documentclass[12pt]{article}%
\usepackage{amsmath}
\usepackage{amsfonts}
\usepackage{amssymb}
\usepackage{graphicx}%
\providecommand{\U}[1]{\protect\rule{.1in}{.1in}}
\newtheorem{theorem}{Theorem}

\newtheorem{definition}[theorem]{Definition}

\newtheorem{proposition}[theorem]{Proposition}
\newtheorem{remark}[theorem]{Remark}

\newenvironment{proof}[1][Proof]{\noindent\textbf{#1.} }{\ \rule{0.5em}{0.5em}}
\begin{document}

\title{Relative periodic solutions of the $n$-vortex problem on the sphere}
\author{Carlos Garc\'{\i}a-Azpeitia \thanks{Depto. Matem\'{a}ticas y Mec\'{a}nica
IIMAS, Universidad Nacional Aut\'{o}noma de M\'{e}xico, Apdo. Postal 20-726,
01000 Ciudad de M\'{e}xico, M\'{e}xico. cgazpe@mym.iimas.unam.mx}}
\maketitle

\begin{abstract}
This paper gives an analysis of the movement of $n\ $vortices on the sphere.
When the vortices have equal circulation, there is a polygonal solution that
rotates uniformly around its center. The main result concerns the global
existence of relative periodic solutions that emerge from this polygonal
relative equilibrium. In addition, it is proved that the families of relative
periodic solutions contain dense sets of choreographies.


MSC 34C25, 37G40, 47H11, 54F45

\end{abstract}

\section{Introduction}

The model for the interaction of $n$-vortex points on the sphere is derived in
\cite{Gr52} and \cite{Bo77}. Among other applications, this model describes
the interaction of hurricanes on the surface on the Earth, vortices in fluid
dynamics, Bose-Einstein condensates and semiconductors. A modern exposition of
the general challenges associated to the $n$-vortex problem is presented in
\cite{Ne01} and references therein.

The $n$ vortices on the unitary sphere $S^{2}$ are described by $v_{j}(t)\in
S^{2}$ for $j=1,..,n$, where the equations are
\begin{equation}
\dot{v}_{j}=\sum_{i\neq j}\frac{v_{i}\times v_{j}}{\left\Vert v_{j}%
-v_{i}\right\Vert ^{2}}. \label{eqv}%
\end{equation}
The stereographic projections of the vortices in the complex plane are given
by%
\[
q_{j}=\frac{x_{j}+iy_{j}}{1-z_{j}}\in\mathbb{C},\qquad v_{j}=(x_{j}%
,y_{j},z_{j})~.
\]
On the sphere, the problem has the solution $q_{j}(t)=e^{i\omega t}\left(
re^{ij\zeta}\right)  $ (polygonal relative equilibrium) when
\[
\omega=\left(  n-1\right)  \frac{1-r^{4}}{8r^{2}},\qquad\zeta=\frac{2\pi}{n}.
\]

The nonlinear stability of this polygonal relative equilibrium has been
established in \cite{Mo11} and \cite{Bo03}. In \cite{Mo11}, a constant
background vorticity is imposed on the sphere so that the total vorticity adds
to zero, equation (\ref{eqv}), while \cite{Bo03} considers an additional
vortex fixed at one of the poles. On the other hand, the existence of families
of relative periodic solutions that bifurcate from the polygonal relative
equilibrium in the plane is proved in \cite{GaIz12}. The main purpose of this
work is to extend the results in \cite{GaIz12} to the sphere.

In a rotating frame, the normal frequencies $\nu_{k}$ of the polygonal
equilibrium on the sphere are given in equation (\ref{vk}) for $k=1,...,n$.
Since $\nu_{k}=\nu_{n-k}$, the normal frequencies are $1:1$ resonant except
for $k=n/2$ for $n$ even, which is also the case of the polygonal relative
equilibrium in the plane \cite{GaIz12}. Due to these resonances, the Lyapunov
center theorem cannot be used to prove the existence of these families.
Actually, in \cite{Ca14} the local existence of relative periodic solutions is
proved using Lyapunov's center theorem, but only for the non-resonant case
$k=n/2$ for $n=2,4,6$. On the other hand, the local existence of periodic
solutions can be obtained for the resonant frequencies $\nu_{k}$ with the
equivariant Weinstein-Moser theorem \cite{Mo88} or Fadell-Rabinowitz theorem
\cite{Bar93}, but those theorems do not prove that the periodic solutions form
a continuous family. We prove the following,

\begin{theorem}
\label{1}For each $k\in\{1,...,n-1\}$ if $n\leq7$ and $k\in\{1,2,n-2,n-1\}$ if
$n\geq8$ such that
\[
4r^{2}\left(  1+r^{2}\right)  ^{-2}<2-\frac{k(n-k)}{n-1},
\]
there is a global continuous family of relative periodic solutions of the
form
\[
q_{j}(t)=e^{i\omega t}x_{j}(\nu t)\text{,}%
\]
that emerges from the polygonal relative equilibrium $x_{j}=re^{ij\zeta}$ with
initial frequency $\nu=\nu_{k}$, where $x_{j}(t)\in\mathbb{C}$ are $2\pi
$-periodic functions satisfying the symmetries
\[
x_{j}(t)=e^{ij\zeta}x_{n}(t+jk\zeta)\text{.}%
\]

\end{theorem}

We use the change of variables developed in \cite{GaIz11} to obtain the
isotypic decomposition. With this change of variables, we decompose the
Hessian in $2\times2$ blocks. Our computations give similar results to\ the
computations in \cite{Bo03} and \cite{Mo11}, where the nonlinear stability of
the polygonal relative equilibrium is analyzed. As a side consequence of our
approach, we recover the results about the linear stability of the polygonal
relative equilibrium on the sphere.

The bifurcation of periodic solutions is represented by a $G$-equivariant
problem; where the group%
\[
G=\mathbb{Z}_{n}\times SO(2)\times S^{1}%
\]
acts by permuting vortices, rotating positions and translating time. Theorem
\ref{1} is proved by using a $G$-equivariant degree theory that is orthogonal
to the generators of the group, developed in \cite{IzVi03} and \cite{BaKrSt06}%
. This approach is analogous to the one implemented for the plane in
\cite{GaIz12}. The \emph{global property} of the families means that the
periodic solutions $x=(x_{1},...,x_{n})$ form a continuum which has Sobolev
norm or period going to infinity, that either ends in a collision orbit or in
a bifurcation point of different equilibrium.

These results can be easily extended to vortices in any radially symmetric
surface, or to bodies on the sphere \cite{Di}. However, the analysis of the
spectrum for bodies is considerably more difficult to perform; refer to
\cite{GaIz13} for bodies on the plane.

A \emph{choreography} is a periodic solution of the $n$-vortex problem with
$n$ equal vortices following the same path. An analogous result for the
existence of choreographies along families of relative periodic solutions that
emerge from the polygonal relative equilibrium for the $n$ body problem, given
in \cite{Ch09} and \cite{CaDoGa}, is proven in Section 4, where we show that a
relative periodic solution for vortices is a choreography when
\[
\nu=\frac{m}{\ell}\omega~,
\]
with $\ell$ and $m$ are relatively prime such that $k\ell-m\in n\mathbb{Z}$.
These conditions on the frequency are satisfied for a dense set of parameters
$\nu$. In fact, choreographies that emanate from the polygonal relative
equilibrium in the plane and sphere have been computed numerically in
\cite{CaDoGa2}. Choreographies for $4$ vortices on the sphere have been found
in \cite{Bo05}

In Section 2 we define the equivariant properties of the bifurcation operator
and present a reduction to a finite number of Fourier components. In Section 3
we analyze the spectra for the isotypic components. In Section 4 we prove
Theorem \ref{1}. In Section 5 we prove that families of relative periodic
solutions contain choreographies.

\section{Setting up the problem}

From \cite{Mo13}, the system of equations for $n$ vortices on the sphere,
parameterized by the stereographic projection, has Hamiltonian%
\[
H(q)=-\frac{1}{2}\sum_{i<j}\ln\frac{\left\vert q_{j}-q_{i}\right\vert ^{2}%
}{\left(  1+\left\vert q_{j}\right\vert ^{2}\right)  \left(  1+\left\vert
q_{i}\right\vert ^{2}\right)  },
\]
and symplectic form,%
\[
\Omega=\sum_{j=1}^{n}\frac{2i}{\left(  1+\left\vert q_{j}\right\vert
^{2}\right)  ^{2}}dq_{j}\wedge d\bar{q}_{j}=\sum_{j=1}^{n}\frac{4}{\left(
1+\left\vert q_{j}\right\vert ^{2}\right)  ^{2}}dx_{j}\wedge dy_{j}\text{.}%
\]

Let the amended potential be
\[
V(q)=\omega G(q)+H(q),\qquad G(q)=2\sum_{j=1}^{n}\frac{\left\vert
q_{j}\right\vert ^{2}}{1+\left\vert q_{j}\right\vert ^{2}}.
\]
In rotating coordinates, $q_{j}(t)=e^{i\omega t}u_{j}(t)$, the equation for
$n$ vortices on the sphere is given by%
\begin{equation}
4\left(  1+\left\vert u_{j}\right\vert ^{2}\right)  ^{-2}i\dot{u}_{j}%
=\nabla_{u_{j}}V(u)\text{,} \label{Sp}%
\end{equation}
where%
\[
\nabla_{u_{j}}V(u)=4\left(  1+\left\vert u_{j}\right\vert ^{2}\right)
^{-2}\omega u_{j}-\sum_{i=1(i\neq j)}^{n}\left(  \frac{u_{j}-u_{i}}{\left\vert
u_{j}-u_{i}\right\vert ^{2}}-\frac{u_{j}}{1+\left\vert u_{j}\right\vert ^{2}%
}\right)
\]

Let%
\[
s_{1}=\sum_{j=1}^{n-1}\frac{1-e^{ij\zeta}}{\left\vert 1-e^{ij\zeta}\right\vert
^{2}}=\frac{n-1}{2}\text{.}%
\]
Equation (\ref{Sp}) for the $n$ vortices on the sphere has the polygonal equilibrium%

\begin{equation}
\mathbf{a}=(a_{1},...,a_{n}),\qquad a_{j}=re^{ij\zeta}~,
\end{equation}
when%
\[
\omega=\frac{1}{4}\left(  1+r^{2}\right)  ^{2}\left(  \frac{s_{1}}{r^{2}%
}-\left(  n-1\right)  \frac{1}{1+r^{2}}\right)  =s_{1}\left(  \frac{1-r^{4}%
}{4r^{2}}\right)  ~.
\]

Changing variables by $x(t)=u(t/\nu)$ (hereafter we use real coordinates for
$x$), the $2\pi/\nu$-periodic solutions of the differential equation for the
vortices are solutions of equation
\[
4\left(  1+\left\vert x_{j}\right\vert ^{2}\right)  ^{-2}\nu J\dot{x}%
_{j}=\nabla_{x_{j}}V(x).
\]
Let $\Psi=\{x\in\mathbb{R}^{2n}:x_{i}=x_{j}\}$ be the set of collision points.
Since $H_{2\pi}^{1}(\mathbb{R}^{2n})\subset C_{2\pi}^{0}(\mathbb{R}^{2n})$, we
define the open set of collision-free $2\pi$-periodic paths as%
\begin{equation}
H_{2\pi}^{1}(\mathbb{R}^{2n}\backslash\Psi)=\{x\in H_{2\pi}^{1}(\mathbb{R}%
^{2n}):x_{i}(t)\neq x_{j}(t)\}~.
\end{equation}
Then, $2\pi/\nu$-periodic solutions correspond to zeros of the bifurcation
operator%
\[
f(x,\nu):H_{2\pi}^{1}(\mathbb{R}^{2n}\backslash\Psi)\times\mathbb{R}%
^{+}\rightarrow L_{2\pi}^{2},
\]
with components%
\begin{equation}
f_{j}(x)=-\nu4\left(  1+\left\vert x_{j}\right\vert ^{2}\right)  ^{-2}J\dot
{x}_{j}+\nabla_{x_{j}}V(x).
\end{equation}

\begin{remark}
The analysis of bifurcation of periodic solutions on the sphere is represented
by an operator defined in the manifold $H_{2\pi}^{1}(\left(  S^{2}\right)
^{n}\backslash\Psi)$ with $\Psi$ the collision set. The stereographic
projection gives a chart of this manifold: $H_{2\pi}^{1}(\left(
\mathbb{R}^{2}\right)  ^{n}\backslash\Psi)\hookrightarrow H_{2\pi}^{1}(\left(
S^{2}\right)  ^{n}\backslash\Psi)$. In this paper we prove the global property
of the family only in the chart $H_{2\pi}^{1}(\left(  \mathbb{R}^{2}\right)
^{n}\backslash\Psi)$. While these results can be extended to consider the
global bifurcation in the whole manifold $H_{2\pi}^{1}(\left(  S^{2}\right)
^{n}\backslash\Psi)$, this would require to extend the definition of
equivariant degree to manifolds, a program that is out of the scope of our presentation.
\end{remark}

By using the stereographic projection of the sphere, we can analyze the
existence of periodic solutions of the $n$ vortices on the sphere analogously
to the plane \cite{GaIz12}.

\begin{definition}
Let $\zeta$ be the permutation of $\{1,...,n\}$ given by $\zeta(j)=j+1$
modulus $n$. We define the action of $\left(  \zeta,\theta\right)
\in\mathbb{Z}_{n}\times SO(2)$ in $x\in H_{2\pi}^{1}(\left(  \mathbb{R}%
^{2}\right)  ^{n}\backslash\Psi)$ by
\begin{equation}
\rho(\zeta)x_{j}=x_{\zeta(j)},\qquad\rho(\theta)x_{j}=e^{-J\theta}%
x_{j}\text{,}%
\end{equation}
and the action of $\varphi\in S^{1}$ by
\begin{equation}
\rho(\varphi)x(t)=x(t+\varphi).
\end{equation}

\end{definition}

The potential $V$ is invariant under the action of $\mathbb{Z}_{n}\times
SO(2)$ because the vortices have the same circulation. Then the gradient
$\nabla V$ is $\mathbb{Z}_{n}\times SO(2)$-equivariant. Since the equation is
autonomous, the operator $f$ is $G$-equivariant under the action of the
abelian group $G=\mathbb{Z}_{n}\times SO(2)\times S^{1}$.

The infinitesimal generators of groups $S^{1}$ and $SO(2)$ are%
\[
\text{ }Ax=\frac{\partial}{\partial\varphi}|_{\varphi=0}x(t+\varphi)=\dot
{x}\text{ and }A_{1}x=\frac{\partial}{\partial\theta}|_{\theta=0}%
e^{-\mathcal{J\theta}}x=-\mathcal{J}x\text{,}%
\]
where
\[
\mathcal{J}=\mathrm{diag}(J,...,J).
\]
Since $V$ is $SO(2)$-invariant, the gradient $\nabla V(x)$ must be orthogonal
to the generator $A_{1}x$, $\left\langle \nabla V,\mathcal{J}x\right\rangle
_{\mathbb{R}^{2n}}=0$. The operator $f$ is $G$-orthogonal to the generators
due to equalities%
\[
\left\langle f(x),\dot{x}\right\rangle _{L_{2\pi}^{2}}=-\nu\sum_{j=1}%
^{n}\left\langle \frac{1}{4}\left(  1+\left\vert x_{j}\right\vert ^{2}\right)
^{2}J\dot{x}_{j},\dot{x}_{j}\right\rangle _{L_{2\pi}^{2}}+V(x(t))|_{0}^{2\pi
}=0\text{,}%
\]
and
\[
\left\langle f(x),\mathcal{J}x\right\rangle _{L_{2\pi}^{2}}=-\nu G\left(
x\right)  |_{0}^{2\pi}+\int_{0}^{2\pi}\left\langle \nabla V,\mathcal{J}%
x\right\rangle ~dt=0,
\]
because%

\[
\partial_{t}G(x)=\sum_{j=1}^{n}4\left(  1+\left\vert x_{j}\right\vert
^{2}\right)  ^{-2}\left\langle x_{j},\dot{x}_{j}\right\rangle ~.
\]

\begin{remark}
Unlike the case of the plane \cite{GaIz12}, the operator $f$ cannot be
represented as the gradient of a functional on the sphere. Then the operator
$f$ is a genuine example of an orthogonal operator that is not the gradient of
a functional. The Hopf fibration can be used to lift the Hamiltonian system
from the sphere $S^{2}$ to $S^{3}$ such that the system in $S^{3}$ has a
Lagrangian representation \cite{Va13}, i.e. the operator in $S^{3}$ has a
gradient representation. However, the analysis of the gradient in $S^{3}$ has
additional challenges.
\end{remark}

Define $\mathbb{\tilde{Z}}_{n}$ as the subgroup generated by $(\zeta,\zeta
)\in\mathbb{Z}_{n}\times S^{1}$ with $\zeta=2\pi/n\in S^{1}$. Since the action
of $(\zeta,\zeta)$ leaves the equilibrium $\mathbf{a}$ fixed, the isotropy
group of $\mathbf{a}$\textbf{ }is%
\[
G_{\mathbf{a}}=\mathbb{\tilde{Z}}_{n}\times S^{1}\text{.}%
\]
Thus, the orbit of $\mathbf{a}$ is isomorphic to the group $G/G_{\mathbf{a}%
}=SO(2)$. In fact, the orbit consists of rotations of the equilibrium. As a
consequence, the generator of the orbit $A_{1}\mathbf{a}=-\mathcal{J}%
\mathbf{a}$ must be in the kernel of $D^{2}f(\mathbf{a})$.

In order to apply the orthogonal degree developed in \cite{IzVi03}, we need to
make a reduction of the bifurcation operator to finite space. Let $Px$ be the
projection $Px=\sum_{\left\vert l\right\vert \leq p}x_{l}e^{ilt}$. Define
$x_{1}=Px$, $x_{2}=(I-P)x$, $f_{1}=Pf$ and $f_{2}=(I-P)f$. Since
$\partial_{x_{2}}f_{2}(x_{1}+x_{2};\nu)$ is uniformly invertible in a closed
bounded set inside $H_{2\pi}^{1}(\mathbb{R}^{2n}\backslash\Psi)\times
\mathbb{R}^{+}$, we can solve $x_{2}(x_{1},\nu)$ from $f_{2}(x_{1}+x_{2}%
;\nu)=0$ by using the global implicit function theorem in this closed bounded
set, see \cite{GaIz12} for details.

The bifurcation function
\[
\phi(x_{1};\nu)=f_{1}(x_{1},x_{2}(x_{1},\nu);\nu),
\]
has the same equivariant and orthogonal properties that $f$ (Chapter 1 in
\cite{IzVi03}). Moreover, the linearization of the bifurcation function at
equilibrium $\mathbf{a}$ is%
\[
\phi^{\prime}(\mathbf{a})x_{1}=\sum_{\left\vert l\right\vert \leq p}\left(
-4\left(  1+r^{2}\right)  ^{-2}\nu(li\mathcal{J)}+D^{2}V(\mathbf{a})\right)
x_{l}e^{ilt}\text{.}%
\]
Since the bifurcation operator is real, the linearization of the bifurcation
function is determined by the blocks $M(l\nu)$ for $l\in\{0,...,p\}$, where
$M(\nu)$ is the matrix%

\begin{equation}
M(\nu)=-4\left(  1+r^{2}\right)  ^{-2}\nu(i\mathcal{J)}+D^{2}V(\mathbf{a}%
)\text{.} \label{M}%
\end{equation}
The matrix $M(l\nu)$ represents the $l$th Fourier component of the linearized
equation at equilibrium $\mathbf{a}$.

\section{Irreducible representations}

To compute the Hessian of $V$, we define $A_{ij}$ as the $2\times2$ minors of
$D^{2}V($\textbf{$a$}$)\in M_{\mathbb{R}}(2n)$, i.e.,
\[
D^{2}V(\mathbf{a})=(A_{ij})_{ij=1}^{n}\text{.}%
\]

\begin{proposition}
We have for $j\in\{1,...,n-1\}$ that%
\[
A_{nj}=\frac{1}{\left(  2r\sin(j\zeta/2)\right)  ^{2}}\left(  e^{jJ\zeta
}R\right)  ~,\qquad A_{nn}=A-\sum_{j=0}^{n-1}A_{nj}~,
\]
where $R=\mathrm{diag}(1,-1)$ and
\[
A=s_{1}r^{-2}I-4s_{1}\left(  1+r^{2}\right)  ^{-2}\mathrm{diag}(1,0)~.
\]

\end{proposition}

\begin{proof}
Notice that $\ $%
\[
\nabla_{x_{n}}V(x)=4\omega\frac{x_{n}}{\left(  1+\left\vert x_{n}\right\vert
^{2}\right)  ^{2}}+(n-1)\frac{x_{n}}{1+\left\vert x_{n}\right\vert ^{2}}%
-\sum_{i\neq n}\frac{x_{n}-x_{i}}{\left\vert x_{n}-x_{i}\right\vert ^{2}%
}\text{.}%
\]
Therefore, the minor $A_{nj}$ is%
\[
A_{nj}=\left.  D_{x_{j}}\nabla_{x_{n}}V(x)\right\vert _{x_{j}=a_{j}}=\left.
D_{x_{j}}\left(  -\frac{x_{n}-x_{j}}{\left\vert x_{n}-x_{j}\right\vert ^{2}%
}\right)  \right\vert _{x_{j}=a_{j}}\text{.}%
\]
Given that $a_{j}=(r\cos j\zeta,r\sin j\zeta)$, then%
\[
A_{nj}=\frac{1}{\left\vert a_{n}-a_{j}\right\vert ^{2}}I-\frac{2r^{2}%
}{\left\vert a_{n}-a_{j}\right\vert ^{4}}\left(
\begin{array}
[c]{cc}%
(1-\cos j\zeta)^{2} & -(1-\cos j\zeta)\sin j\zeta\\
-(1-\cos j\zeta)\sin j\zeta & (\sin j\zeta)^{2}%
\end{array}
\right)  .
\]
Since%
\[
\left\vert a_{n}-a_{j}\right\vert ^{2}=2r^{2}(1-\cos j\zeta)=4r^{2}\sin
^{2}(j\zeta/2)\text{,}%
\]
and $\sin^{2}j\zeta=(1-\cos j\zeta)(1+\cos j\zeta)$, then%
\[
A_{nj}=\frac{1}{4r^{2}\sin^{2}(j\zeta/2)}\left(  I-\left(
\begin{array}
[c]{cc}%
1-\cos j\zeta & -\sin j\zeta\\
-\sin j\zeta & 1+\cos j\zeta
\end{array}
\right)  \right)  \text{.}%
\]

Let
\[
A=D_{x_{n}}\left.  \left(  4\omega\frac{x_{n}}{\left(  1+\left\vert
x_{n}\right\vert ^{2}\right)  ^{2}}+(n-1)\frac{x_{n}}{1+\left\vert
x_{n}\right\vert ^{2}}\right)  \right\vert _{x_{n}=a_{n}},
\]
then the matrix $A_{nn}$\ satisfies%
\[
A_{nn}=A-\sum_{j\neq n}D_{x_{n}}\left(  -\frac{x_{n}-x_{j}}{\left\vert
x_{n}-x_{j}\right\vert ^{2}}\right)  =A-\sum_{j\neq n}A_{nj}\text{.}%
\]
It only remains to compute%
\[
A=\left(  \frac{4\omega}{\left(  1+r^{2}\right)  ^{2}}+\frac{n-1}{1+r^{2}%
}\right)  I-2r^{2}\left(  \frac{8\omega}{\left(  1+r^{2}\right)  ^{3}}%
+\frac{n-1}{\left(  1+r^{2}\right)  ^{2}}\right)  \mathrm{diag}(1,0)\text{.}%
\]
By the definition of $\omega$, we have
\[
4\omega\left(  1+r^{2}\right)  ^{-2}+(n-1)\left(  1+r^{2}\right)  ^{-1}%
=\frac{n-1}{2}r^{-2}~,
\]
and
\[
8\omega\left(  1+r^{2}\right)  ^{-3}+(n-1)\left(  1+r^{2}\right)
^{-2}=(n-1)r^{-2}\left(  1+r^{2}\right)  ^{-2}\text{.}%
\]
Then $A=s_{1}r^{-2}I-4s_{1}\left(  1+r^{2}\right)  ^{-2}\mathrm{diag}(1,0).$
\end{proof}

In order to apply the orthogonal degree we need to find the isotypic
decomposition for the action of $\mathbb{\tilde{Z}}_{n}<\mathbb{Z}_{n}\times
SO(2)$.

\begin{definition}
For $k\in\{1,...,n\}$, we define isomorphisms $T_{k}:\mathbb{C}^{2}\rightarrow
W_{k}$ by%
\[
T_{k}(w)=(n^{-1/2}e^{(ikI+J)\zeta}w,...,n^{-1/2}e^{n(ikI+J)\zeta}w),
\]
where
\[
W_{k}=\{(e^{(ikI+J)\zeta}w,...,e^{n(ikI+J)\zeta}w):w\in\mathbb{C}^{2}%
\}\subset\mathbb{C}^{2n}.
\]

\end{definition}

In \cite{GaIz11} is proved that subspaces $W_{k}$ are the isotypic components
under the action of $\mathbb{\tilde{Z}}_{n}$, where the action of
$(\zeta,\zeta)\in\mathbb{\tilde{Z}}_{n}$ on the space $W_{k}$ is
\[
\rho(\zeta,\zeta)=e^{ik\zeta}\text{.}%
\]
Since subspaces $W_{k}$ are orthogonal, the linear transformation%
\[
Pw=\sum_{k=1}^{n}T_{k}(w_{k}),\qquad w=(w_{1},...,w_{n})\in\mathbb{C}^{2n},
\]
is orthogonal. Given that $P$ rearranges the coordinates of the isotypic
decomposition, we have, from Schur's lemma,%
\[
P^{-1}D^{2}V(\mathbf{a})P=\mathrm{diag}(B_{1},...,B_{n})\text{,}%
\]
where $B_{k}$ are matrices which satisfy $D^{2}V(\mathbf{a})T_{k}%
(w)=T_{k}(B_{k}w)$.

In Proposition 7 of \cite{GaIz11} is found that the blocks $B_{k}$ are given
by%
\begin{equation}
B_{k}=\sum_{j=1}^{n}A_{nj}e^{j(ikI+J)\zeta}\text{ for }k\in\{1,...,n\}\text{.}%
\end{equation}
This formula has been successfully applied to study the bifurcation of
periodic solutions from the polygonal equilibrium for bodies and vortices in
\cite{GaIz13} and \cite{GaIz12}, respectively. In the plane, for the polygonal
relative equilibrium with radius $r=1$, the computation in \cite{GaIz12} shows
that $B_{k}=s_{1}I+\left(  s_{1}-s_{k}\right)  R$, where%
\begin{equation}
s_{k}:=\frac{k(n-k)}{2}.
\end{equation}
In the next proposition, we calculate $B_{k}$ for the case of $n$ vortices on
the sphere.

\begin{proposition}
We have%
\[
B_{k}=s_{1}r^{-2}I+\left(  s_{1}-s_{k}\right)  r^{-2}R-4s_{1}\left(
1+r^{2}\right)  ^{-2}\mathrm{diag}(1,0)~.
\]

\end{proposition}

\begin{proof}
Since $A_{nn}=A-\sum_{j=0}^{n-1}A_{nj}$, then%
\[
B_{k}=A+\sum_{j=1}^{n-1}A_{nj}(e^{j(ikI+J)\zeta}-I).
\]
Since $A_{nj}=\left(  2r\sin(j\zeta/2)\right)  ^{-2}\left(  e^{jJ\zeta
}R\right)  $, then
\[
B_{k}=A+\sum_{j=1}^{n-1}\frac{1}{\left(  2r\sin(j\zeta/2)\right)  ^{2}%
}(Re^{jik\zeta}-e^{jJ\zeta}R).
\]
Using the equalities $e^{-(jJ\zeta)}+e^{(jJ\zeta)}=2I\cos j\zeta$ and
$e^{jik\zeta}+e^{-jik\zeta}=2\cos jk\zeta$, we conclude that
\[
B_{k}=A+\sum_{j=1}^{n-1}\frac{\cos jk\zeta-\cos j\zeta}{\left(  2r\sin
(j\zeta/2)\right)  ^{2}}R.
\]
Since $4\sin^{2}(j\zeta/2)=2(1-\cos j\zeta)$, then%
\[
B_{k}=A+\sum_{j=1}^{n-1}\frac{-2\sin^{2}(jk\zeta/2)+2\sin^{2}(j\zeta
/2)}{\left(  2r\sin(j\zeta/2)\right)  ^{2}}R~.
\]
The result follows from equality%
\[
s_{k}:=\frac{k(n-k)}{2}=\frac{1}{2}\sum_{j=1}^{n-1}\frac{\sin^{2}(jk\zeta
/2)}{\sin^{2}(j\zeta/2)}~\text{.}%
\]

\end{proof}

For the $0$th Fourier component, the real matrix $M(0)=D^{2}V(\mathbf{a})\in
M_{\mathbb{R}}(2n\mathbb{)}$ is equivalent to the matrix%
\[
\mathrm{diag}(B_{1},...,B_{n/2},B_{n})\text{, }%
\]
where $B_{k}\in M_{\mathbb{C}}(2\mathbb{)}$ for $k\in\lbrack1,n/2)\cap
\mathbb{N}$ and $B_{n/2},B_{n}\in M_{\mathbb{R}}(2\mathbb{)}$.

For the $l$th Fourier components, the matrix $M(l\nu)$ in the new coordinates
is given by%
\[
P^{-1}M(l\nu)P=\mathrm{diag}(m_{1}(l\nu),...,m_{n}(l\nu))\text{,}%
\]
where blocks $m_{k}(\nu)\in M_{\mathbb{C}}(2\mathbb{)}$ are%
\[
m_{k}(\nu)=-\nu(iJ\mathcal{)}4\left(  1+r^{2}\right)  ^{-2}+B_{k}\text{.}%
\]

The action of $(\zeta,\zeta,\varphi)\in\mathbb{\tilde{Z}}_{n}\times S^{1}$ in
$W_{k}$ for the $l$th Fourier mode is given by
\[
\rho(\zeta,\zeta,\varphi)w_{k}=e^{ik\zeta}e^{il\varphi}w_{k}.
\]

\section{Bifurcation theorem}

\begin{proposition}
For
\begin{equation}
4r^{2}\left(  1+r^{2}\right)  ^{-2}<2-s_{k}/s_{1}, \label{in}%
\end{equation}
we define%
\begin{equation}
\nu_{k}=\frac{\left(  1+r^{2}\right)  ^{2}}{4r^{2}}\left[  s_{k}\left(
2s_{1}-s_{k}-4s_{1}r^{2}\left(  1+r^{2}\right)  ^{-2}\right)  \right]  ^{1/2}.
\label{vk}%
\end{equation}
Therefore, the Morse index $n_{k}(\nu)$ of $m_{k}(\nu)$ changes at $\nu_{k}$
and%
\[
\eta_{k}(\nu_{k})=\lim_{\varepsilon\rightarrow0}\left(  n_{k}(\nu
_{k}-\varepsilon)-n_{k}(\nu_{k}+\varepsilon)\right)  =-1.
\]

\end{proposition}

\begin{proof}
The determinant of
\[
m_{k}(\nu)=\left(
\begin{array}
[c]{cc}%
\left(  2s_{1}-s_{k}\right)  r^{-2}-4s_{1}\left(  1+r^{2}\right)  ^{-2} &
i4\nu\left(  1+r^{2}\right)  ^{-2}\\
-i4\nu\left(  1+r^{2}\right)  ^{-2} & s_{k}r^{-2}%
\end{array}
\right)  \text{,}%
\]
is%
\[
d_{k}(\nu)=-\nu^{2}16\left(  1+r^{2}\right)  ^{-4}+r^{-4}s_{k}\left(
2s_{1}-s_{k}-4s_{1}r^{2}\left(  1+r^{2}\right)  ^{-2}\right)  .
\]
This determinant changes sign only at values $\pm\nu_{k}$ when (\ref{in})
holds. The trace of $m_{k}(\nu)$ is%
\[
T_{k}(\nu)=2s_{1}r^{-2}-4s_{1}\left(  1+r^{2}\right)  ^{-2}=2s_{1}\left(
r^{4}+1\right)  r^{-2}\left(  1+r^{2}\right)  ^{-2}>0.
\]
Since $T_{k}(0)>0$ and $d_{k}(0)>0$ when (\ref{in}) holds, the Morse index at
$\nu=0$ is $n_{k}(0)=0$. Since $d_{k}(\infty)<0$, then $n_{k}(\infty)=1$.
Given that $n_{k}(\nu)$ changes only at the positive value $\nu_{k}$, then
\[
\eta(\nu_{k})=n_{k}(0)-n_{k}(\infty)=-1.
\]

\end{proof}

Note that $2-s_{k}/s_{1}\geq0$ for all $k\in\{1,...,n-1\}$ if $n\leq7$ and for
$k\in\{1,2,n-2,n-1\}\ $if $n\geq8$, i.e. the frequencies $\nu_{k}$ are real
only in these cases. We say that the frequency $\nu_{k}$ is
\emph{non-resonant} if $l\nu_{k}\neq\nu_{j}$ for $j\neq k$ and $l\geq2$. The
proof of Theorem \ref{1} is consequence of the following theorem.

\begin{theorem}
Assume that $r$ satisfies $4r^{2}\left(  1+r^{2}\right)  ^{-2}\neq
2-s_{k}/s_{1}$ for all $k\in\{1,...,n-1\}$. For each $k\in\{1,...,n-1\}$ such
that $\nu_{k}$ is non-resonant, the operator $f(x,\nu)$ has a global
bifurcation of zeros $(x,\nu)$ from $\left(  \mathbf{a},\nu_{k}\right)  $ with
symmetries%
\[
x_{j}(t)=e^{j\zeta J}x_{n}(t+jk\zeta)~.
\]

\end{theorem}

\begin{proof}
The linear map $\phi^{\prime}(\mathbf{a};\nu)$ has $l$th Fourier components
$M(l\nu)$ with blocks $m_{k}(l\nu)$ for $l=0,...,p$. Let $M^{\perp}(0)$ be the
matrix $M(0)=D^{2}V($\textbf{$a$}$)$ in the orthogonal complement to the
generator of the orbit $A_{1}$\textbf{$a$}$=-\mathcal{J}$\textbf{$a$}. Then
$M^{\perp}(0)$ is invertible by the hypothesis that $4r^{2}\left(
1+r^{2}\right)  ^{-2}\neq2-s_{k}/s_{1}$, i.e. $G$\textbf{$a$} is an hyperbolic
orbit. In the new coordinates $w$, the vector $\mathcal{J}\mathbf{a}$
corresponds to $w_{n}=e_{2}=(0,1)$. Then $B_{n}$ in the orthogonal complement
of $e_{2}$ is given by $e_{1}^{T}B_{n}e_{1}=2s_{1}r^{-2}-4s_{1}\left(
1+r^{2}\right)  ^{-2}$, which is positive. Therefore, $\sigma:=\mathrm{sign}%
~e_{1}^{T}B_{n}e_{1}=1$.

Since $\nu_{k}$ is non-resonant, then $M(l\nu)$ is invertible at $\nu=\nu_{k}$
for $l=2,...,p$. For the $1$th Fourier mode, the Morse index of $m_{k}%
(\nu)=\left.  M(\nu)\right\vert _{W_{k}}$ changes at $\nu_{k}$ with $\eta
_{k}(\nu_{k})=-1$, and the action of $(\zeta,\zeta,\varphi)\in G_{\mathbf{a}}$
in $W_{k}$ is given by $\rho(\zeta,\zeta,\varphi)w_{k}=e^{ik\zeta}e^{i\varphi
}w_{k}$. Then the isotropy group of the space $W_{k}$ is the group
$\mathbb{\tilde{Z}}_{n}(k)$ generated by $\left(  \zeta,\zeta,-k\zeta\right)
\in G_{\mathbf{a}}$,
\[
\mathbb{\tilde{Z}}_{n}(k)=\left\langle \left(  \zeta,\zeta,-k\zeta\right)
\right\rangle .
\]
Using the computations for equivariant orthogonal degree in Proposition 3.2 of
Chapter 4 in \cite{IzVi03}, we conclude that the $G$-orthogonal degree in a
neighborhood of $\left(  x_{1},\nu\right)  =(G\mathbf{a},\nu_{k})$ is
\[
\mathrm{deg}\left(  d\left(  x_{1},G\mathbf{a}\right)  -\varepsilon,\phi
(x_{1};\nu)\right)  =\eta_{k}(\nu_{k})\left(  \mathbb{\tilde{Z}}%
_{n}(k)\right)  +\eta_{n-k}(\nu_{k})\left(  \mathbb{\tilde{Z}}_{n}%
(n-k)\right)  +...\text{ ,}%
\]
where the sum includes non-maximal isotropy groups $H$. Since $\eta_{k}%
(\nu_{k})=-1\neq0$, by the existence property of the degree, we conclude the
local existence of two solutions with isotropy groups $\mathbb{\tilde{Z}}%
_{n}(k)$ and $\mathbb{\tilde{Z}}_{n}(n-k)$. A solution with isotropy group
$\mathbb{\tilde{Z}}_{n}(k)$ satisfies symmetries%
\[
x_{j}(t)=\rho\left(  \zeta,\zeta,-k\zeta\right)  x_{j}=e^{-J\zeta}%
x_{j+1}(t-k\zeta)\text{.}%
\]

The global property follows from assuming that the branch is contained in a
closed bounded set inside $H_{2\pi}^{1}(\mathbb{R}^{2}\backslash\Psi
)\times\mathbb{R}^{+}$, unless it is a continuum set with period or Sobolev
norm going to infinite or with a collision orbit. Applying orthogonal degree
to the reduced function $\phi(x_{1})$ in this compact set, we conclude that
the sum of the local degrees at bifurcation points is zero as in Theorem 5.2,
Chapter 2 in \cite{IzVi03}. Since all the degrees have indices $\eta_{k}%
(\nu_{k})=-1$, the sum of the degrees can never be zero unless it goes to a
different equilibrium.
\end{proof}

By the definition (\ref{vk}), we have that $l\nu_{k}\neq\nu_{j}$ if and only
if%
\[
4r^{2}\left(  1+r^{2}\right)  ^{-2}=-\frac{\left(  s_{j}\left(  s_{j}%
-2s_{1}\right)  -l^{2}s_{k}\left(  s_{k}-2s_{1}\right)  \right)  }{\left(
s_{1}s_{j}-l^{2}s_{1}s_{k}\right)  }\rightarrow2-s_{k}/s_{1}%
\]
when $l\rightarrow\infty$. Therefore, the frequency $\nu_{k}$ is non-resonant
except for a countable set of parameters $4r^{2}\left(  1+r^{2}\right)  ^{-2}$
that accumulates at $2-s_{k}/s_{1}$. For the resonant frequencies the theorem
proves the existence of a branch with the higher frequency $\nu_{j}=l\nu_{k}$.

The frequency $\nu_{n}=0$ corresponds to rotations of the equilibrium and is
not a bifurcation point. The existence of periodic solutions cannot be proved
at $r_{k}$ with $4r_{k}^{2}\left(  1+r_{k}^{2}\right)  ^{-2}=2-s_{k}/s_{1}$
because $G$\textbf{$a$} is not an hyperbolic orbit. Actually, since $\det
B_{k}$ changes sign at $r_{k}$, there is a bifurcation of relative equilibria
at $r_{k}$, see \cite{Mo13} or \cite{GaIz11}.

The condition for linear stability is that all the normal frequencies are
real, i.e., that inequality (\ref{in}) holds for all $k\in\{1,...,n-1\}$. In
polar coordinates $r^{2}=\left(  1+\cos\theta\right)  \left(  1-\cos
\theta\right)  ^{-1}$, where $\theta$ is the polar angle, the linear stability
condition is equivalent to%
\[
\frac{k(n-k)}{n-1}-1<\cos^{2}\theta\text{ for }k\in\{1,...,n-1\}\text{.}%
\]
Using the fact that $k(n-k)$ is increasing in $k$ for $k\in\{1,...,[n/2]\}$,
we obtain a similar result that in \cite{Bo03} and \cite{Mo11}.

From the equality $B_{n-k}=B_{k}$, we have $m_{n-k}(\nu)=\bar{m}_{k}(-\nu)$.
This equality is a consequence of the fact that the problem is equivariant
under the action of the non-abelian group $G\cup\kappa G$, where $\kappa$ is a
reflection that acts as $\rho(\kappa)x_{j}(t)=\bar{x}_{n-j}(-t)$. Therefore,
it is possible that the families of relative periodic solutions with
symmetries $\mathbb{\tilde{Z}}_{n}(k)$ and $\mathbb{\tilde{Z}}_{n}(n-k)$ are
related by the reflection $\kappa$.

\section{Choreographies}

The term \textit{choreography} was adopted for the $n$-body problem in
celestial mechanics after the work of Sim\'{o} \cite{Si00}. The first
non-circular choreography was discovered numerically for $3$ bodies in
\cite{Mo93}, and its existence was proved analytically in \cite{ChMo00} using
the direct method of calculus of variations.

For $4$ vortices in the plane, choreographies have been found in \cite{Bo04}
and \cite{Bo05}. Recently, choreographies have also been constructed for $n$
vortices in general bounded domains with blow-up techniques. They are located
close to stagnation points of one vortex in the domain \cite{Ba16}, and close
to its boundary \cite{Ba18}.

Let
\[
Q_{j}(t):=q_{j}(t/\nu)=e^{it\omega/\nu}x_{j}(t),~
\]
be a reparametrization of the positions of the vortices in the inertial
reference frame. The following theorem shows that the families of relative
periodic solutions contain dense sets of choreographies.

\begin{theorem}
\label{proposition} A relative periodic solution with symmetry group
$\mathbb{\tilde{Z}}_{n}(k)$ and frequency $\nu=\omega(m/\ell)$, where $\ell$
and $m$ are relatively primes such that $k\ell-m\in n\mathbb{Z}$, satisfies%
\[
Q_{j}(t)=Q_{n}(t+j\tilde{k}\zeta)~\text{,}%
\]
where $\tilde{k}=k-(k\ell-m)\ell^{\ast}$, with $\ell^{\ast}$ the $m$-modular
inverse of $\ell$. Furthermore, the choreography $Q_{n}(t)$ is $2\pi
m$-periodic and symmetric with respect to $2\pi/m$ rotations.
\end{theorem}

\begin{proof}
Since $e^{it\omega/\nu}=e^{it\ell/m}$ is $2\pi m$-periodic, the function
$Q_{n}(t)=e^{it\omega/\nu}x_{n}(t)$ is $2\pi m$-periodic. Furthermore, since%
\begin{equation}
Q_{n}(t-2\pi)=e^{i(t-2\pi)\omega/\nu}x_{n}(t-2\pi)=e^{-i2\pi\ell/m}Q_{n}(t),
\label{symq}%
\end{equation}
the orbit of $Q_{n}(t)$ is invariant under $2\pi/m$ rotations.

The solutions satisfy
\begin{align*}
Q_{j}(t)  &  =e^{it(\omega/\nu)}x_{j}(t)=e^{it(\omega/\nu)}e^{ij\zeta}%
x_{n}(t+jk\zeta)\\
&  =e^{it(\omega/\nu)}e^{ij\zeta}e^{-i(\omega/\nu)\left(  t+jk\zeta\right)
}Q_{n}(t+jk\zeta)=e^{-ij\left(  (\omega/\nu)k-1\right)  \zeta}Q_{n}%
(t+jk\zeta)~\text{.}%
\end{align*}
Using the fact that $\omega/\nu=\ell/m$, $\zeta=2\pi/n$, we have
\[
j\left(  (\omega/\nu)k-1\right)  \zeta=2\pi j\left(  \frac{\ell k-m}%
{mn}\right)  =2\pi j(r/m)~,
\]
with $r=(k\ell-m)/n\in\mathbb{Z}$ by assumption.

Since $\ell$ and $m$ are relatively prime, we can find $\ell^{\ast}$, the
$m$-modular inverse of $\ell$. Since $\ell\ell^{\ast}=1$ mod $m$, it follows
from the symmetry (\ref{symq}) that
\[
Q_{n}(t-2\pi jr\ell^{\ast})=e^{-i2\pi j(r/m)}Q_{n}(t).
\]
Using the fact that $\tilde{k}=k-rn\ell^{\ast}=k-(k\ell-m)\ell^{\ast}$, the
result follows from
\begin{equation}
Q_{j}(t)=e^{-i2\pi j(r/m)}Q_{n}(t+jk\zeta)=Q_{n}(t+j(k-rn\ell^{\ast})\zeta).
\end{equation}

\end{proof}

\noindent\textbf{Acknowledgements.} The author is grateful to T. Bartsch, W.
Krawcewicz, E. Doedel and R. Calleja and the referees for their useful
comments. This project is supported by PAPIIT-UNAM grant IN115019.

\end{document}